\documentclass[11pt]{article}
\usepackage{amssymb}
\usepackage{amsmath}
\usepackage{amsthm}
\usepackage{subfigure}
\usepackage{verbatim}
\usepackage{geometry}
\usepackage{fancyhdr}
\usepackage{enumerate}
\usepackage{graphicx}
\usepackage{amsfonts}
\usepackage{tabularx}
\usepackage{longtable}
\usepackage{multirow}
\newcommand{\floor}[1]{{\left\lfloor #1 \right\rfloor}}
\newtheorem{theorem}{Theorem}

\newtheorem{definition}[theorem]{Definition}

\theoremstyle{definition}
\newtheorem{remark}[theorem]{Remark}

\begin{document}
\title{New  bounds of degree-based topological indices for  some classes of $c$-cyclic graphs }
\author{\emph{Monica Bianchi $^{1}$} \thanks{%
e-mail: monica.bianchi@unicatt.it} \and \emph{Alessandra Cornaro $^{1}$}
\thanks{%
e-mail: alessandra.cornaro@unicatt.it} \and \emph{Jos\'{e} Luis Palacios $%
^{2}$} \thanks{%
e-mail: jopalal@gmail.com} \and \emph{Anna Torriero$^{1}$} \thanks{%
e-mail: anna.torriero@unicatt.it} }
\date{\today}

\maketitle

\centerline{$^{{}^1}$ Department of Mathematics and Econometrics,  Catholic University, Milan,
Italy.}

\centerline{$^{{}^2}$ Department of Electrical and Computer Engineering,}
\centerline{University of New Mexico, Albuquerque, New Mexico.}

\begin{abstract}

Making use of a majorization technique for a suitable class of graphs,  we derive upper and
lower bounds for some topological indices depending on the degree sequence
over all vertices,  namely the first general
Zagreb index and the first multiplicative Zagreb index. Specifically, after
characterizing $c-$cyclic graphs $(0\leq c\leq 6)$ as those whose degree
sequence belongs to particular subsets of $\mathbb{R}^{n}$, we identify the
maximal and minimal vectors of these subsets with respect to the
majorization order. This technique allows us to determine lower and upper
bounds of the above indices recovering  those existing in the literature
as well obtaining new ones.


\end{abstract}

\section{Introduction}

\bigskip\ Many topological indices in Mathematical Chemistry are based on
the degree sequence of a finite graph $G=(V,E)$ over all vertices. One of the most famous
among these is the \textit{first Zagreb index }defined as $M_{1}(G)=\underset%
{i=1}{\overset{n}{\sum }}d_{i}^{2}$ where $d_{i}$ $(i=1,...,n)$ stands for
the degree of the vertex $i$  and $n=\left\vert V\right\vert $\ (see \cite{Gut72}, \cite{Gut75}, \cite{Nikolic}).
The notion of $M_{1}(G)$ was extended by Li and
Zheng \cite{LiZheng} as the \textit{first general Zagreb index} $%
M_{1}^{\alpha }(G)=\underset{i=1}{\overset{n}{\sum }}d_{i}^{\alpha }$,
for $\alpha $ an arbitrary real number different from $0$ and $1,$ that
coincides \ with the zeroth-order general Rand\'{\i}c index (see \cite{LISHI08}). For $\alpha =2$ we recover the first Zagreb index  while
for $\alpha =-1$ we get the \textit{inverse degree }$\rho \left( G\right) =$ $%
M_{1}^{-1}=\underset{j=1}{\overset{n}{\sum }}\dfrac{1}{d_{j}}$ which has
generated increased attention motivated by conjectures of the computer
program Graffiti (see  \cite{Fajtlowicz}).

In this paper we are concerned precisely with
those indices depending on the degree sequence over all vertices of $G$,  for
which we adopt a unified approach aimed to determine new lower and upper
bounds. This fruitful methodology, synthetically introduced  in Section 2, is  based on the majorization order and Schur-convexity (\cite {Marshall}), and has already been used by some of the authors (\cite{BT} and \cite{Grassi}) in other contexts,  as well as for localizing some  relevant topological indicators of a graph (\cite{BCT1}, \cite{BCT2}, \cite{BCPT1} and \cite{BCPT2} ), which is also the aim of the present article.
We restrict our attention to a particular class of graphs, the $c$-cyclic graphs
for $0\leq c\leq 6$, which contain exactly $c$ independent cycles (i.e., cycles that do not contain other cycles within themselves). In Section 3 we provide a new characterization of  $c$-cyclic
graphs,  needed to determine their extremal degree sequences with respect to the majorization order discussed in Section  4.  In Section 5 we determine  upper and lower bounds for some degree-based topological indices. Section  6 concludes with a summary and some final comments.

\section{Notations and preliminaries results on majorization}
In this section we recall some basic notions on majorization, referring for more details  to  \cite{BCT1} and \cite{Marshall}. In the sequel we denote
by $[x_1^{\alpha_1}, x_2^{\alpha_1}, \cdots, x_p^{\alpha_p}]$ a vector in $\mathbb{R}^{n}$ with $\alpha_i$ components equal to $x_i$, where $\sum_{i=1}^p \alpha_i=n$. If $\alpha_i=1$ we use for convenience $x_i$ instead of $x_i^{1}$, while $x_i^0$ means that the component $x_i$ is not present.

\begin{definition}
Given two vectors $\mathbf{y}$, $\mathbf{z\in }$ $D=\{\mathbf{x}\in \mathbb{R%
}^{n}:x_{1}\geq x_{2}\geq ...\geq x_{n}\}$, the majorization order $\mathbf{y%
}\trianglelefteq \mathbf{z}$ means:
\begin{equation*}
\begin{cases}
\left\langle \mathbf{y},\mathbf{s}^{\mathbf{k}}\right\rangle \leq
\left\langle \mathbf{z},\mathbf{s}^{\mathbf{k}}\right\rangle ,\text{ }%
k=1,...,(n-1) \\[3mm]
\left\langle \mathbf{y},\mathbf{s}^{\mathbf{n}}\right\rangle =\left\langle
\mathbf{z},\mathbf{s}^{\mathbf{n}}\right\rangle
\end{cases}
\end{equation*}
where $\left\langle \cdot ,\cdot \right\rangle $ is the inner product in $
\mathbb{R}^{n}$ and
$\mathbf{s^{j}}=[1^j,0^{n-j}],\quad j=1,2,\cdots ,n.$
\end{definition}

In what follows we will consider some subsets of
\begin{equation*}
\Sigma _{a}=D\cap \{\mathbf{x}\in \mathbb{R}_{+}^{n}:\left\langle \mathbf{x},
\mathbf{s^{n}}\right\rangle =a\},
\end{equation*}
where $a$ is a positive real number.  Given a closed subset $S\subseteq \Sigma
_{a} $, a vector $\mathbf{x}^{\ast }(S)\in S$ is said to be maximal for $S$
with respect to the majorization order if $\mathbf{x}\trianglelefteq \mathbf{
x}^{\ast }(S)$ for each $\mathbf{x}\in S$. Analogously, a vector $\mathbf{x}
_{\ast }(S)\in S$ is said to be minimal for $S$ with respect to the
majorization order if $\mathbf{x}_{\ast }(S)\trianglelefteq \mathbf{x}$ for
each $\mathbf{x}\in S$. Notice that if $S\subseteq T$, then $\mathbf{x^{\ast }} (S)\trianglelefteq \mathbf{
x}^{\ast }(T)$ and $\mathbf{x}_{\ast }(T) \trianglelefteq \mathbf{x}_{\ast }(S)$.

\noindent In \cite{BCT1} some of the authors derived the maximal and minimal
elements, with respect to the majorization order, of the set
\begin{equation}
S_{a}=\Sigma _{a}\cap \left\{ \mathbf{x\in }\text{ }%
%TCIMACRO{\U{211d} }%
%BeginExpansion
\mathbb{R}
%EndExpansion
^{n}:M_{i}\geq x_{i}\geq m_{i},\text{ }i=1,...n\right\} ,  \label{Sa}
\end{equation}
where $\mathbf{m}=\left[ m_{1},m_{2},...,m_{n}\right]$ and $\mathbf{M}=%
\left[ M_{1},M_{2},...,M_{n}\right]$ are two fixed vectors arranged
in nonincreasing order with $0\leq m_{i}\leq M_{i},$ for all $i=1,...n,$ and
$a$ is a positive real number such that $\left\langle \mathbf{m},\mathbf{s}^{%
\mathbf{n}}\right\rangle \leq a\leq \left\langle \mathbf{M},\mathbf{s}^{%
\mathbf{n}}\right\rangle .$
For the sake of completeness we recall the main results we will use in Section 4. We start discussing the maximal element. Let $\mathbf{v^j}=[0^j,1^{n-j}]$, $j=0, \cdots, n$.
\begin{theorem}\label{th:massimo}
Let $k\geq 0$ be the smallest integer such that
\begin{equation}
\left\langle \mathbf{M},\mathbf{s}^{\mathbf{k}}\right\rangle +\left\langle
\mathbf{m},\mathbf{v}^{\mathbf{k}}\right\rangle \leq a<\left\langle \mathbf{M%
},\mathbf{s}^{k+1}\right\rangle +\left\langle \mathbf{m},\mathbf{v}^{\mathbf{%
k+1}}\right\rangle ,  \label{dis1}
\end{equation}%
and $\theta =a-\left\langle \mathbf{M},\mathbf{s}^{\mathbf{k}}\right\rangle
-\left\langle \mathbf{m},\mathbf{v}^{\mathbf{k+1}}\right\rangle .$ Then
\begin{equation}\label{xsa}
\mathbf{x}^{\ast }(S_{a})=[M_1,M_2,\cdots, M_k, \theta, m_{k+2}, \cdots m_n].
\end{equation}
\end{theorem}

From this general result, the maximal element of particular
subsets of $S_{a}$ can be deduced. In what follows we will often focus on sets of the type
\begin{equation*}
S_{a}^{\left[ h\right] }=
\Sigma _{a}\cap \left\{ \mathbf{x}\in \mathbb{R}^{n}:M_{1}\geq x_{1}\geq
...\geq x_{h}\geq m_{1},\right.
\left. M_{2}\geq x_{h+1}\geq ...\geq x_{n}\geq m_{2}\right\}
\end{equation*}
where $1\leq h\leq n$, $0\leq m_{2}\leq m_{1}$, $0\leq M_{2}\leq M_{1}$, $m_{i}<M_{i},i=1,2$
and $$hm_{1}+(n-h)m_{2}\leq a\leq hM_{1}+(n-h)M_{2}.$$

In this case,  given $a^{\ast }=hM_{1}+(n-h)m_{2}$, let
\begin{equation*}
k=\left\{
\begin{array}{ccc}
\Big\lfloor {\dfrac{a-h(m_{1}-m_{2})-nm_{2}}{M_{1}-m_{1}}} \Big\rfloor &
\text{ if } & a<a^{\ast } \\
&  &  \\
\Big\lfloor {\dfrac{a-h(M_{1}-M_{2})-nm_{2}}{M_{2}-m_{2}}} \Big\rfloor &
\text{ if } & a\geq a^{\ast }
\end{array}
\right.
\end{equation*}
where $\floor{x}$ denote the integer part of the real number $x$. In Corollary 3 in \cite{BCT1} it has been   shown that
\begin{equation*}
\mathbf{x}^{\ast }(S_{a}^{\left[ h\right] })=
\left\{
\begin{array}{ccc}
\left [M_1^k, \theta, m_1^{h-k-1},m_2^{n-h} \right]  & \text{ if } & a<a^{\ast } \\
&  &  \\
\left [ M_1^h,M_2^{k-h}, \theta, m_2^{n-k-1} \right ] & \text{ if } & a\geq a^{\ast }%
\end{array}%
\right.
\end{equation*}
where $\theta$ is evaluated in order to entail $\mathbf{x}^{\ast }(S_{a}^{\left[ h\right] }) \in \Sigma_a$.

The computation of the minimal element of the set $S_a$ is more tangled. The minimal element of $\Sigma_a$ is $x_*(\Sigma_a)= [(\frac an)^n]$. If it belongs to $S_a$ then it is its  minimal element, too. Otherwise we will use the following theorem
\begin{theorem}
\label{th:minimo}  Let $k \ge 0$ and $d \ge 0$ be the smallest
integers such that
\begin{itemize}
\item[1)] $k + d <n $
\item[2)] $m_{k+1} \le \rho \le M_{n-d}$ where $\rho= \dfrac { a - \langle
\mathbf{m},\mathbf{s^k} \rangle - \langle \mathbf{M}, \mathbf{v^{n-d}}
\rangle} {n-k-d}$.
\end{itemize}
Then
\begin{equation*}
\mathbf{x_{\ast }}(S_{a})=[m_1, \cdots , m_k, \rho^{n-d-k}, M_{n-d+1} \cdots, M_n].
\end{equation*}
\end{theorem}
For the set $S_a^{[h]}$ we can express the minimal element in a more accessible way (see Corollary 10 in \cite{BCT1}). We recall only the expression for the case $m_1 \le M_2$, since it is the only one  we need in the sequel:

%\begin{equation}
%\mathbf{x_{\ast }}(S_{a}^{\left[ h\right] })=\left\{
%\begin{array}{ccc}
%\dfrac{a}{n}\mathbf{s^{n}} & \text{ if } & m_{1}\leq \dfrac{a}{n}\leq M_{2}
%\\
%& & \\
%m_{1}\mathbf{s^{h}}+\dfrac{a-hm_{1}}{n-h}\mathbf{v^{h}} & \text{ if } &
%\dfrac{a}{n}<m_{1} \\
%& & \\
%\dfrac{a-M_{2}(n-h)}{h}\mathbf{s^{h}}+M_{2}\mathbf{v^{h}} & \text{ if } &
%\dfrac{a}{n}>M_{2} \\
%&  &
%\end{array}%
%\right.  \label{m_1 le M_2}
%\end{equation}

\begin{equation}\label{m_1 le M_2}
\mathbf{x_{\ast }}(S_{a}^{\left[ h\right] })=\left\{
\begin{array}{ccc}
\left [(\frac{a}{n})^n \right ] & \text{ if } & m_{1}\leq \dfrac{a}{n}\leq M_{2}
\\
& & \\
\left [m_{1}^h,\left (\frac{a-hm_{1}}{n-h} \right )^{n-h}\right ] & \text{ if } &
\dfrac{a}{n}<m_{1} \\
& & \\
\left [\left (\frac{a-M_{2}(n-h)}{h} \right )^h, M_{2}^{n-h} \right ] & \text{ if } &
\dfrac{a}{n}>M_{2} \\
&  &
\end{array}%
\right.
\end{equation}

Notice that  the minimal element of the set $
S_{a}$ does not necessarily have integer components,
while this is not the case for the maximal element. For our purposes
it is crucial to find the minimal vector in $S_{a}$
with integer components which can be  constructed by the following procedure (see  Remark 12 in \cite{BCT1}).
Let us consider, for instance,  the vector $x_{\ast }(S_{a}^{\left[ h\right] })=%
\left [ \left (\dfrac{a}{n} \right )^n \right ] $ which corresponds to the case $m_{1}\leq \dfrac{a%
}{n}\leq M_{2}$. If $\dfrac{a}{n}$ is not an
integer, we will find the index $k$, $1\leq k\leq n$, such that
\begin{equation*}
\left(\Big\lfloor{ \dfrac{a}{n}}\Big\rfloor +1\right)k+\Big\lfloor{ \dfrac{a}{n}}\Big\rfloor (n-k)=a
\end{equation*}
i.e., $k=a-\Big\lfloor{ \dfrac{a}{n}}\Big\rfloor\ n$. The vector
\begin{equation*}
\mathbf{x_{\ast }^{1}}=\left [ \left (\Big\lfloor{ \dfrac{a}{n}}\Big\rfloor +1\right )^k, \left( \Big\lfloor{ \dfrac{a}{n}}\Big\rfloor \right )^{n-k} \right]
\end{equation*}%
is the minimal element of $S_{a}^{\left[ h\right] }$ with integer components. With slight modifications the same procedure can be applied also to the other cases discussed in \eqref{m_1 le M_2} of Theorem \ref{th:minimo}.

\section{ A new characterization of $c$-cyclic graphs}

Let $G=(V,E)$  a simple, connected, undirected graph with order $|V|=n$  and size $|E|=m$. Denote by $\pi = (d_1,d_2, \cdots, d_n)$ the degree sequence of $G$, where $d_i$ is the degree of  the vertex $i$, arranged in nonincreasing order $d_1 \ge d_2 \ge \cdots \ge d_n$. It is well known that $$\sum {d_i}=2m, \,\, \text{ and } d_1 \leq n-1 \le m. $$

The cyclomatic number $c$ of a graph $G$ is given by $c=m-n+1$. It corresponds to the number of independent cycles in $G$, i.e. cycles that do not contain other cycles   (see \cite{BOL90}). In particular, graphs  with cyclomatic number $c=0$ are  trees, graphs  with cyclomatic number $c=1$ are unicyclic graphs and, more generally, graphs  with cyclomatic number $c$ are  $c$-cyclic graphs. In this Section we will deal with graphs having a cyclomatic number $c \le 6$.

Schocker in \cite{Scho}  gave a characterization of the degree sequences of c-cyclic graphs (see also \cite{Gut89},  \cite{Tome2008},\cite{Zhu2011},  \cite{Zhang2006}). For the sake of completeness we recall his results for the case $0\le c \le 6$ (see Theorem 2.4 and Corollary 2.5 in \cite{Scho}).

\begin{theorem} \label{Schocker}
The integers $(n-1) \ge d_1 \ge d_2 \ge \cdots \ge d_n $ are the vertex degree sequence of
\begin{enumerate}[i.]
\item a tree ($c=0$) if and only if $m \ge 1$, $n=m+1$;
\item a unicyclic graph ($c=1$) if and only if $m \ge 3$, $n=m$,  $d_1+d_2 \le n+1$;
\item a bicyclic graph ($c=2$) if and only if $m \ge 5$, $n=m-1$,  $d_1+d_2 \le n+2$, $d_1+d_2+d_3 \le n+4$;
\item a tricyclic graph  ($c=3$) if and only if $m \ge 6$, $n=m-2$,
$d_1+d_2 \le n+3$ and $d_1+d_2+d_3 \le n+5$;
   \item a tetracyclic  graph ($c=4$) if and only if $m\geq 8$, $n=m-3$,  $d_{1}+d_{2}\leq n+4$ , $d_{1}+d_{2}+d_{3}\leq n+6$ and $%
d_{1}+d_{2}+d_{3}+d_{4}\leq n+9$;
\item a pentacyclic graph ($c=5$) if and only if $m\geq 9$, $n=m-4$,  $d_{1}+d_{2}\leq n+5$ , $d_{1}+d_{2}+d_{3}\leq n+7$, $%
d_{1}+d_{2}+d_{3}+d_{4}\leq n+10,\,\, 2d_{1}+2d_{2}+d_{3}+d_{4}+d_{5}\leq 2n+16.$
\item a hexacyclic graph ($c=6$) if and only if $m\geq 10$, $n=m-5$,  $d_{1}+d_{2}\leq n+6$ , $d_{1}+d_{2}+d_{3}\leq n+8$, $%
d_{1}+d_{2}+d_{3}+d_{4}\leq n+11$,\,\, $2d_{1}+2d_{2}+d_{3}+d_{4}+d_{5}\leq 2n+18$,\,\,  $2d_{1}+2d_{2}+d_{3}+d_{4}+d_{5}+d_6\leq 2n+20$.
\end{enumerate}
\end{theorem}

\noindent In the class of $c$-cyclic graphs we are interested in finding  graphs associated to the maximal  (minimal) degree sequence
with respect to the majorization order.
 %order applying  the results of Section 2.

To this aim, by Theorem \ref{Schocker}, we derive a new characterization of c-cyclic graphs as those whose degree sequences belongs to particular subsets of $\Sigma_{2(n+c-1)}$.
    Next theorem gives the structure of the subset of $\Sigma_{2(n+c-1)}$ we will deal with. Notice that the cases $c \le 2$ are well known (see \cite{Tome2008}) while the case $c=3$ was incorrectly discussed in Lemma 2.2 of \cite{Zhu2011}.

\begin{theorem} \label{Schocker1}
The integers $(n-1) \ge d_1 \ge d_2 \ge \cdots \ge d_n $ are the vertex degree sequence of
\begin{enumerate}[i.]
\item a tree ($c=0$) if and only if $n \ge 2$, $\sum_{i=1}^n d_i= 2(n-1)$;
\item a unicyclic graph ($c=1$) if and only if $n \ge 3$, $\sum_{i=1}^n d_i= 2n$ and at least three of them are greater or equal to 2;
\item a bicyclic graph ($c=2$) if and only if $n\ge 4$, $\sum_{i=1}^n d_i= 2(n+1)$ and at least four  of them are greater than or equal to 2;
\item a tricyclic graph  ($c=3$) if and only if $n \ge 4$, $\sum_{i=1}^n d_i= 2(n+2)$, and one of the following conditions holds:
\begin{enumerate}[1.]
\item if $n \ge 5$, at least five of them are greater or equal to 2,
\item at least four of them are greater or equal to 3;
\end{enumerate}
\item a tetracyclic  graph ($c=4$) if and only if $n \ge 5$, $\sum_{i=1}^n d_i= 2(n+3)$, and one of the following conditions holds:
\begin{enumerate}[1.]
\item if $n \ge 6$, at least six of them are greater than or equal to 2,
\item at least four of them are greater or equal to 3 and at least five of them are greater than or equal to 2;
\end{enumerate}
\item a pentacyclic graph ($c=5$) if and only if $n \ge 5$, $\sum_{i=1}^n d_i= 2(n+4)$, and one of the following conditions holds:
\begin{enumerate}[1.]
\item if $n \ge 7$, at least seven of them are greater than or equal to 2,
\item if $n \ge 6$ at least six of them are greater than or equal to 2 and at least four of them are greater than or equal to 3,
    \item at least five of them are greater than or equal to 3 and at least three of them are greater than or equal to 4.
\end{enumerate}
\item a hexacyclic graph ($c=6$) if and only if $n \ge 5$, $\sum_{i=1}^n d_i= 2(n+5)$, and one of the following conditions holds:
\begin{enumerate}[1.]
\item if $n \ge 8$, at least eight of them are greater than or equal to 2,
\item if $n \ge 7$ at least seven of them are greater than or equal to 2 and at least four of them are greater than or equal to 3,
\item if $n \ge 6$ at least six of them are greater than or equal to 2,  at least five of them are greater than or equal to 3 and at least three of them are greater or equal to 4,
\item if $n \ge 6$ at least six of them are greater or equal to 3,
\item at least five of them are greater than or equal to 4.
\end{enumerate}
\end{enumerate}
\end{theorem}

\begin{proof}

{\bf  Case $\mathbf{c=0}$}: well known.

{\bf Case $\mathbf{c=1}$}. {\it Necessary condition}:  it is evident that the vertices of the unique cycle have the degrees greater than or equal to  2 (see also Lemma 2.1 in \cite{Tome2008}).

{\it Sufficient condition}:  the condition  $d_1+d_2\le n+1$ is equivalent to   $d_3+ \cdots + d_n \ge n-1$ and this inequality is satisfied if $d_3 \ge 2$.

{\bf Case $\mathbf{c=2}$}.   {\it Necessary condition}:  let us assume by contradiction that $d_4 <2$, i.e. $d_4=1$ which implies  $d_1+d_2+d_3 = n+5$,  against the  condition $d_1+d_2+d_3 \le n+4$ (see also Lemma 2.2 in \cite{Tome2008}).

{\it Sufficient condition}: the condition  $d_1+d_2+d_3 \le n+4$ is  equivalent to   $d_4+\cdots + d_n \ge n-2$ while the condition $d_1+d_2 \le n+2$ is equivalent to  $d_3+\cdots+  d_n \ge n$. If $d_4 \ge 2$ they are both satisfied

{\bf Case $\mathbf{c=3}$}. For $n=4$ there is only one tricyclic graph associated to the degree sequence $[3^4]$.  Let us consider the case  $n \ge 5$.

{\it Necessary condition}:   assume by contradiction that $d_5=1$ and $d_4 \le 2$. This implies  $d_1+d_2+d_3+d_4=n+8$ and  $d_1+d_2+d_3 \ge  n+6$, against  the condition $d_1+d_2+d_3 \le n+5$.

{\it Sufficient condition}:  the condition $d_1+d_2+d_3 \le n+5$ is equivalent to   $d_4+\cdots + d_n \ge n-1$ while the condition $d_1+d_2 \le n+3$ is  equivalent
to $d_3+\cdots+  d_n \ge n+1 $.
  %  If $n=4$ this inequality is satisfied when $d_4 \ge 3$, and the unique admissible degree sequence is which satisfies also the second inequality.
They are both  satisfied if  either $d_5 \ge 2$ or $d_4 \ge 3$.

{\bf Case $\mathbf{c=4}$}:
%For $n=5$ the condition $d_{1}+d_{2}+d_{3}+d_{4}\leq 14$ is equivalent to  $d_{5} \geq 2.$ The second condition $d_{1}+d_{2}+d_{3}\leq 11 $ is equivalent to  $d_{4}+d_{5}\geq 5$    which is
%satisfied  if at least $d_{5}\geq 2$ and $d_4 \ge 3$. Finally  $d_{1}+d_{2}\leq 9$   is equivalent to   $d_{3}+d_4+d_{5}\geq 7$  which is fulfilled if  $d_{5}\geq 2$ and $d_4 \ge 3$. On the other hand, if the condition is not satisfied, then $d_5=1$ and $d_4 \le 2$. Thus $d_1+d_2+d_3 \ge 13$  contradicting the necessary condition $d_1+d_2+d_3 \le 11$.
Let us consider the case $n \ge 6$.

{\it Necessary condition}:  assume by contradiction that both the conditions are not fulfilled. This means that either $d_5=1$ or $d_6=1$ and $d_4 \le 2$.  In the first case $d_1+d_2+d_3+d_4= n+10$ contradicts the necessary condition $d_1+d_2+d_3+d_4\le n+9$. In the second case, $d_1+d_2+d_3+d_4+d_5=n+11$ and  $d_1+d_2+d_3 \ge n+7$ contradict the necessary condition $d_1+d_2+d_3 \le n+6$.

{\it Sufficient condition}: the condition $d_{1}+d_{2}+d_{3}+d_{4}\leq n+9$ is equivalent to  $d_{5}+\cdots +d_{n}\geq n-3$ which is satisfied if $d_{5}\geq 2$,  thus in particular if either \emph{v.1} or \emph{v.2} hold.

The  condition $d_{1}+d_{2}+d_{3}\leq n+6$ is equivalent to  $d_{4}+\cdots +d_{n}\geq n$   while  $%
d_{1}+d_{2}\leq n+4$   is equivalent to   $d_{3}+\cdots +d_{n}\geq n+2$ . They are both
satisfied  if either \emph{v.1} or \emph{v.2} hold.
%either if $d_{6}\geq 2$ or $d_{5}\geq 2$ and $d_4 \ge 3$.

The case $n=5$ can be proved following similar steps.

{\bf Case $\mathbf{c=5}$}: for $n=5$ there is only one pentacyclic graph associated to the degree sequence $[4^3,3^2]$.

Let us consider the case $n \ge 7$.

{\it Necessary condition}: assume by contradiction that neither of the three conditions are fulfilled. Then one of the following cases should occur:

a)\,\,  $d_6=1$ and $d_5 \le 2$;\,\,\,\, b) \,\, $d_6=1$ and $d_3 \le 3$; \,\,\,\, c) \,\, $d_7=1$ and $d_4 \le 2$.

 %
% \begin{enumerate}[a)]
% \item $d_6=1$ and $d_5 \le 2$,
% \item $d_6=1$ and $d_3 \le 3$,
% \item $d_7=1$ and $d_4 \le 2$,
% %\item $d_7=1, d_4 \le 2$ and $d_3 \le 3$.
% \end{enumerate}

 In  cases a) and b) we have $d_1+d_2+d_3+d_4+d_5=n+13$. In case a) we get $d_1+d_2+d_3+d_4 \ge n+11$ which contradicts the necessary condition $d_1+d_2+d_3+d_4 \le n+10$. In case b) $d_1+d_2 \ge n+4$ implies  $2d_1+2d_2 \ge 2n+8$. By using the necessary condition  $2d_1+2d_2+d_3+d_4+d_5 \le 2n+16$, we get $d_3+d_4+d_5 \le 8$ and thus $d_5 \le 2$. But under this condition we get again $d_1+d_2+d_3+d_4\ge n+11$ contradicting the necessary condition  $d_1+d_2+d_3+d_4\le n+10$.

 Finally in case c)  we have $d_1+d_2+d_3+d_4+d_5+d_6=n+14$ and  $d_1+d_2+d_3 \ge n+8$ against the necessary condition $d_1+d_2+d_3 \le n+7$.

{\it Sufficient condition}: the condition $d_{1}+d_{2}+d_{3}+d_{4}\leq n+10$  is equivalent to
$d_5 + \cdots + d_n \ge n-2$ which is
satisfied if either  $d_{6}\geq 2$ or $d_{5}\geq 3$, thus in particular if \emph{vi.1} or \emph{vi.2} or \emph{vi.3} hold.

The conditions $d_{1}+d_{2}+d_{3}\leq n+7$, equivalent to $d_{4}+\cdots
+d_{n}\geq n+1$, and $d_1+d_2 \le n+5$, equivalent to $d_3 + \cdots d_n \ge n+3$, are both fulfilled if \emph{vi.1} or \emph{vi.2} or \emph{vi.3} hold.
%if $d_{7}\geq 2$ or $d_6 \ge 2 $ and $d_4 \ge 3$ or $d_{5}\geq 3$.

The last condition $2d_{1}+2d_{2}+d_{3}+d_{4}+d_{5}\leq 2n+16$ is equivalent to
 $d_{6}+ \cdots ..+d_{n}\geq d_{1}+d_{2}-8$.
 Easy computations show that it is satisfied either for  $d_{7}\geq 2$ or for $d_6 \ge 2 $ and $d_4 \ge 3$. Indeed, for $d_7 \ge 2$ we get
 $$
 d_6+ \cdots + d_n \ge 4 + (n-7) = (n+5) - 8 \ge d_1+d_2 -8,
 $$
 while for  $d_6 \ge 2$ and $d_4 \ge 3$ we get
 $$
 d_3+d_4+d_6+\cdots +d_n \ge 6+2+(n-6)= (n+10)-8 \ge d_1+d_2+d_3+d_4 -8.
 $$
Finally, for  $d_{5}\geq 3$ and $d_3 \ge 4$ we get
$$
d_3+d_4+d_6+\cdots +d_n \ge 7+ (n-5)= (n+10)-8 \ge d_1+d_2+d_3+d_4 -8.
 $$

The case $n=6$ can be proved following similar steps.

 {\bf Case $\mathbf{c=6}$}: for $n=5$ there is only one pentacyclic graph associated to the degree sequence $[4^5]$. Let us  consider the case $n \ge 8.$

  {\it Necessary condition}:  let us assume by contradiction that neither of the  conditions are fulfilled. Then one of the following five cases should occur:

 a) \,\, $d_6=1$ and $d_5 \le 3$; \,\,\,\, b) \,\,$d_6=1$ and $d_4 \le 2$; \,\,\,\, c)  $d_7=1$ and $d_5 \le 2$;

 d) \,\, $d_7=1$, $d_6 \le 2$  and $d_3 \le 3$;\,\,\,\, e) $d_8=1$ and $d_4 \le 2$.

 %\begin{enumerate}[a)]
% \item $d_6=1$ and $d_5 \le 3$,
%  \item $d_6=1$ and $d_4 \le 2$,
% \item $d_7=1$ and $d_5 \le 2$,
% \item $d_7=1$, $d_6 \le 2$  and $d_3 \le 3$,
% \item $d_8=1$ and $d_4 \le 2$.
% \end{enumerate}

 In  cases a) and b) we have $d_1+d_2+d_3+d_4+d_5=n+15$. In case a) we get $d_1+d_2+d_3+d_4 \ge n+12$ against the necessary condition $d_1+d_2+d_3+d_4 \le n+11$. In case b) we get $d_1+d_2+d_3  \ge n+11$ against  the necessary condition $d_1+d_2+d_3 \le n+8$.

 In cases c) and d) we have $d_1+d_2+d_3+d_4+d_5 + d_6 =n+16$. In case c) we get again $d_1+d_2+d_3+d_4 \ge n+12$ which contradicts the necessary condition $d_1+d_2+d_3+d_4 \le n+11$. In case d) we get  $d_1+d_2 \ge n+5$ which implies $2d_1+2d_2 \ge 2n+10$. By using the necessary condition  $2d_1+2d_2+d_3+d_4+d_5 +d_6\le 2n+20$, we get $d_3+d_4+d_5 +d_6\le 10$ and thus $d_5 \le 2$ or $d_6=1$.  Under the  condition $d_5 \le 2$ we go back to case c) while for $d_6=1$ to case a).

 Finally in case e)  we have $d_1+d_2+d_3+d_4+d_5+d_6+d_7=n+17$ and we get $d_1+d_2+d_3 \ge n+9$ contradicting the necessary condition $d_1+d_2+d_3 \le n+8$.

{\it Sufficient condition}: the condition $d_{1}+d_{2}+d_{3}+d_{4}\leq n+11$  is equivalent to
$d_5 + \cdots + d_n \ge n-1$ which is
satisfied if $d_{7}\geq 2$ or $d_6 \ge 2$ and $d_5 \ge 3$ or $d_{5}\geq 4$, thus in particular if \emph{vii.1}-\emph{vii.5}  hold.

The conditions $d_{1}+d_{2}+d_{3}\leq n+7$, equivalent to $d_{4}+\cdots
+d_{n}\geq n+2$, and $d_1+d_2 \le n+5$, equivalent to $d_3 + \cdots d_n \ge n+4$, are both fulfilled if one of the conditions \emph{vii.1}-\emph{vii.5}  holds.
 %$d_{8}\geq 2$ or $d_7 \ge 2 $ and $d_4 \ge 3$ or $d_{6}\geq 2$ and $d_5 \geq 3$ or $d_5 \ge 3$ and $d_4 \ge 4$.

Let us now consider the last two conditions: $2d_{1}+2d_{2}+d_{3}+d_{4}+d_{5}\leq 2n+16$ is equivalent to
 $d_{6}+ \cdots +d_{n}\geq d_{1}+d_{2}-8$ while  $2d_{1}+2d_{2}+d_{3}+d_{4}+d_{5}+d_6 \leq 2n+20$ is equivalent to $d_{7}+ \cdots +d_{n}\geq d_{1}+d_{2}-10$.
First of all, notice that for $d_6 \geq 2$ the first inequality is satisfied if and only if  the second one holds.

 Easy computations show that the inequality $d_{7}+ \cdots +d_{n}\geq d_{1}+d_{2}-10$ is  satisfied either for  $d_{8}\geq 2$ or for $d_7 \ge 2 $ and $d_4 \ge 3$. Indeed, for $d_8 \ge 2$ we get
 $$
 d_7+ \cdots + d_n \ge 4 + (n-8) = (n+6) - 10 \ge d_1+d_2 -10,
 $$
 while for  $d_7 \ge 2$ and $d_4 \ge 3$ we get
 $$
 d_3+d_4+d_6+\cdots +d_n \ge 6+4+(n-7)= (n+11)-8 \ge d_1+d_2+d_3+d_4 -8.
 $$
In a similar way, for $d_6 \geq 2$,  $d_5 \ge 3$ and $d_3 \ge 4$ we get
$$
d_3+d_4+d_6+\cdots +d_n  \ge 4+3+2 +(n-6) = (n+11)-8 \ge d_1+d_2+d_3+d_4 -8
$$
while for $d_6 \ge 3$
$$
d_3+d_4+d_6+\cdots +d_n  \ge 9 + (n-6)= (n+11)-8\ge d_1+d_2+d_3+d_4 -8.
$$
 Finally for $d_5 \ge 4$ we get
 $$
 d_3+d_4+d_7 + \cdots +d_n \geq 8 + (n-6) = (n+12)-10 \geq d_1+d_2+d_3+d_4 -10
 $$
 $$
 d_3+d_4+d_6 + \cdots +d_n \geq 8 + (n-5) = (n+11)-8 \geq d_1+d_2+d_3+d_4 -10
 $$

The case $n=6$ and $n=7$ can be proved following similar steps.
\end{proof}

\section{Extremal degree sequences}

Endowed with Theorem \ref{Schocker1}, we are now in a  position to apply the results of Section 2 in order to get extremal  degree sequences associated to c-cyclic graphs. This topic has been investigated also by Gutman who gave in \cite{Gut89} a characterization of maximal degree sequences by using partitions of an integer in unequal parts (see Lemma 2 in \cite{Gut89}). This approach is completely different from ours and provides only the maximal elements while we are also able to build minimal ones.

\noindent In the sequel, for any $c$,  the vector $\mathbf{d}=(d_1,d_2,\cdots, d_n)$ is arranged in  non increasing order,  $d_1\ge d_2 \ge \cdots \ge d_n$ and $\sum_{i=1}^n d_i =2(n+c-1)$, i.e. $\mathbf{d} \in \Sigma_{2(n+c-1)}$.

\begin{itemize}
\item{\bf Trees}. We face the set
$$
\mathcal{S}_0=\{\mathbf{d} \in D \cap\Sigma_{2(n-1)}:  1 \le d_i\le (n-1) \}
$$
%For the maximal element we apply the frist result recalled above and we have that \\ $k=\Bigg\lfloor \dfrac{2|E|-N}{N-2}\Bigg\rfloor =1$.
%We now compute the value of $\theta$ as follow:
%$$\theta=2|E|-(N-1)k-1(N-k-1)=2(N-1)-(N-1)-(N-2)=1.$$
The maximal element of $\mathcal{S}_0$ is given by:
\begin{equation*}
\mathbf{x^{\ast }}(\mathcal{S}_0)=\left[(n-1), 1^{n-1} \right]
\end{equation*}
which is the degree sequence of the star $S_n$.
%So, the maximum value for the function $\sum {1\over {d_j}}$ is attained at $(N-1)+\dfrac{1}{N-1}.$

The minimal element of the set $\mathcal{S}_0$ is given by
\begin{equation}
\mathbf{x_{\ast} }(\mathcal{S}_0) = \left[ \left (\frac{2(n-1)}{n}\right )^n \right].
\end{equation}
Since it has not integer components for $n>2$,  we  apply the procedure explained in Section 2.
Taking $k= 2(n-1)-\Bigg\lfloor \dfrac{2(n-1)}{n} \Bigg\rfloor n = (n-2)$ for $n\ge 2$, the minimal element with integer component is
\begin{equation*}
\mathbf{x_{\ast }'}(\mathcal{S}_0)=\left[ 2^{n-2}, 1^2 \right]
\end{equation*}
which is the degree sequence of the path $P_n$.

\item{\bf Unicyclic. } We face the set
\begin{align*}
\mathcal{S}_1= \{\mathbf{d} \in \Sigma_{2n}:   1 \le d_n \le \cdots \le d_4 \le (n-1),\,\,\,  2 \le d_3 \le d_2 \le d_1 \le (n-1) \}
\end{align*}
The maximal element of $\mathcal{S}_1$ is given by:
\begin{equation*}
\mathbf{x^{\ast }}(\mathcal{S}_1)=\left[(n-1), 2^2, 1^{n-3} \right]
\end{equation*}
which is the degree sequence of  the graph obtained joining two leaves of the star $S_n$.
The minimal element of the set $\mathcal{S}_1$ is given by
\begin{equation}
\mathbf{x_{\ast} }(\mathcal{S}_1) = \left[2^n \right ]
\end{equation}
which is the degree sequence of the cycle.
\item{\bf Bicyclic. } We face the set
\begin{align*}
\mathcal{S}_2= \{\mathbf{d} \in \Sigma_{2(n+1)}:
 1 \le d_n \le \cdots \le d_5 \le (n-1), \,\,\, 2 \le d_4 \le \cdots \le d_1 \le (n-1) \}
\end{align*}
The maximal element of $\mathcal{S}_2$ is given by:
\begin{equation*}
\mathbf{x^{\ast }}(\mathcal{S}_2)=\left[(n-1), 3,2^2, 1^{n-4} \right]
\end{equation*}
which is the degree sequence of the graph $H^1_n$ in \cite{Zhang2006}.

The minimal element of the set $\mathcal{S}_2$ is given by
\begin{equation}
\mathbf{x_{\ast} }(\mathcal{S}_2) = \left[ \left (\frac{2(n+1)}{n}\right)^n \right]
\end{equation}
Since it has not integer components,  we apply the procedure explained  in Section 2.

Taking  $k= 2(n+1)-\Bigg\lfloor \dfrac{2(n+1)}{n} \Bigg\rfloor n = 2 $, the minimal element with integer component is
\begin{equation*}
\mathbf{x_{\ast }'}(\mathcal{S}_2)=\left[ 3^2,2^{n-2} \right ]
\end{equation*}
which is the degree sequence of the class of graph $G_n^1$ in \cite{Zhang2006}. They can be obtained joining two non adjacent nodes in a circle.

\item{\bf Tricyclic.}  For $n=4$ there is only one tricyclic graph associated to the sequence $[3^4]$.  Assuming $n \ge 5$, unlike the previous cases, we  should now consider  two different sets:
  \begin{align*}
\mathcal{S}_3^1= \{\mathbf{d} \in \Sigma_{2(n+2)}:    1 \le d_n \le \cdots \le d_6 \le (n-1), \,\,\,  2 \le d_5 \le \cdots \le d_1 \le (n-1) \}
\end{align*}
\begin{align*}
\mathcal{S}_3^2= \{\mathbf{d} \in \Sigma_{2(n+2)}:  1 \le d_n \le \cdots \le d_5 \le (n-1), \,\,\, 3 \le d_4 \le \cdots \le d_1 \le (n-1) \}
\end{align*}

The maximal elements are, respectively,
\begin{equation*}
\mathbf{x^{\ast }}(\mathcal{S}_3^1)=\left[(n-1), 4,2^3, 1^{n-5} \right]
\end{equation*}
\begin{equation*}
\mathbf{x^{\ast }}(\mathcal{S}_3^2)=\left[(n-1), 3^3, 1^{n-4} \right]
\end{equation*}
Notice that the two degree sequences are not comparable in the majorization order. Thus in case of tricyclic graphs we have two maximal degree sequences which correspond to the graphs $G^{*}(\mathcal{S}_3^1)$ and $G^{*}(\mathcal{S}_3^2)$ in Figure \ref{fig:1}, drawing for the sake of simplicity for $n=8$ (the second one is  the graph $H_4$ in \cite{Zhu2011}).

The minimal element of the set $\mathcal{S}_3^1$ is given by
\begin{equation}
\mathbf{x_{\ast} }(\mathcal{S}_3^1) = \left[\left (\frac{2(n+2)}{n}\right )^n \right].
\end{equation}
Since it has not integer components for $n \ge 5$,  we  apply the procedure given in Section 2.
Taking  $k= 2(n+2)-\Bigg\lfloor \dfrac{2(n+2)}{n} \Bigg\rfloor n = 4 $, the minimal element with integer component is
\begin{equation*}
\mathbf{x_{\ast }'}(\mathcal{S}_3^1)=\left[ 3^4, 2^{n-4}\right ]
\end{equation*}

The minimal element of the set $\mathcal{S}_3^2$ can be computed via formula \eqref{m_1 le M_2}. In our case $\frac a n =\frac{2(n+2)}{n}$,  $m_1=3$ and $h=4$. Since for $n \ge 5$, we are in the case  $\frac{2(n+2)}{n} < 3$.  Simple computations show that
the  minimal element of the set $\mathcal{S}_3^2$ is the same as before,
i.e. $\mathbf{x_{\ast}' }(\mathcal{S}_3^1)=\mathbf{x_{\ast} }(\mathcal{S}_3^2)$.

Thus in case of tricyclic graphs we have one minimal degree sequence which corresponds to the graph $G_{*}(\mathcal{S}_3^1)$ in Figure \ref{fig:2}, drawing for the sake of simplicity for $n=8$.

\item{\bf Tetracyclic.} For $n \ge 6$ we should consider the two different sets:
\begin{align*}
\mathcal{S}_{4}^{1}=\{\mathbf{d}  \in \Sigma_{2(n+3)}:
 1\leq d_{n}\leq \cdots \leq d_{7}\leq (n-1),\,\,\, 2\leq d_{6}\leq \cdots \leq
d_{1}\leq (n-1)\}
\end{align*}%
\begin{align*}
\mathcal{S}_{4}^{2}=\{\mathbf{d}  \in \Sigma_{2(n+3)}:
 & 1\leq d_{n}\leq \cdots \leq d_{6}\leq (n-1),\,\,\, 2 \leq d_5 \le (n-1),\\
 & 3\leq d_{4}\leq \cdots \leq
d_{1}\leq (n-1)\},
\end{align*}
while for $n=5$ only the second one.

For $n \ge 6$, the maximal elements are respectively,
\begin{equation*}
\mathbf{x^{\ast }}(\mathcal{S}_{4}^{1})=\left[ (n-1),5,2^4, 1^{n-6} \right]
\end{equation*}
\begin{equation*}
\mathbf{x^{\ast }}(\mathcal{S}_{4}^{2})=\left[ (n-1),4,3^2, 2,1^{n-5} \right]
\end{equation*}
Notice that  the maximal element for $\mathcal{S}_{4}^{2}$ has been computed via Theorem \ref{th:massimo} providing  $k=1$.

The two degree sequences are not comparable in the majorization
order. Thus also in case of tetracyclic graphs we have two maximal degree sequences which correspond, to the  graphs $G^{*}(\mathcal{S}_{4}^{1})$ and $G^{*}(\mathcal{S}_{4}^{2})$ in Figure \ref{fig:1} drawing for the sake of simplicity for $n=8$.

The minimal element of the set $\mathcal{S}_{4}^{1}$ for $n \ge 6$ is given by
\begin{equation}
\mathbf{x_{\ast }}(\mathcal{S}_{4}^{1})=\left[\left( \frac{2(n+3)}{n} \right )^n  \right] .
\end{equation}%
For $n=6$ we get $\mathbf{x_{\ast }}(\mathcal{S}_{4}^{1}) = [3^6]$,  while for $n > 6$ the procedure given in
Section 2 must be applied.  Taking $k=2(n+3)-\Bigg\lfloor\dfrac{2(n+3)}{n}\Bigg\rfloor n=6$,
the minimal element with integer components is
\begin{equation*}
\mathbf{x_{\ast }^{\prime }}(\mathcal{S}_{4}^{1})=\left[ 3^6,2^{n-6} \right]
\end{equation*}

%To compute the minimal element of the set  $\mathcal{S}_{4}^{2}$ we should  apply Theorem \ref{th:minimo}, but we guess that the  minimal  element is
%the same as before, i.e. $\mathbf{x_{\ast }'}(\mathcal{S}_{4}^{1})=\mathbf{%
%x_{\ast }}(\mathcal{S}_{4}^{2})$. This can be proved easily by direct computations. Indeed  $\mathbf{x_{\ast }'}(\mathcal{S}_{4}^{1})\in{S}_{4}^{2}$  and it is impossible to build an element of ${S}_{4}^{2}$ minorizing it.

The  minimal  element of the set $\mathcal{S}_{4}^{2}$ for $n \ge 6$ is
the same as before, i.e. $\mathbf{x_{\ast }'}(\mathcal{S}_{4}^{1})=\mathbf{%
x_{\ast }}(\mathcal{S}_{4}^{2})$. This can be proved easily by direct computations. Indeed  $\mathbf{x_{\ast }'}(\mathcal{S}_{4}^{1})\in{S}_{4}^{2}$  and it is impossible to build an element of ${S}_{4}^{2}$ minorizing it.

Thus in case of tetracyclic graphs we have one minimal degree sequence which corresponds the graph
$G_{*}(\mathcal{S}_{4}^{1})$ in  Figure \ref{fig:2} drawing for the sake of simplicity for $n=8$.

Finally  for $n=5$  simple computations give as maximal element $[4^2,3^2,2]$ and as minimal $[4,3^4]$.

\item {\bf Pentacyclic.} For $n \ge 7$ we should  consider three different sets:
\begin{align*}
\mathcal{S}_{5}^{1}=\{\mathbf{d} \in \Sigma_{2(n+4)}:
1\leq d_{n}\leq \cdots \leq d_{6}\leq (n-1),\,\,\, 2\leq d_{7}\leq \cdots \leq
d_{1}\leq (n-1)\}
\end{align*}%
\begin{align*}
\mathcal{S}_{5}^{2}=\{\mathbf{d} \in  \Sigma_{2(n+4)}: \,\,
&  1\leq d_{n}\leq \cdots \leq d_{7}\leq (n-1),\,\,\,  2 \le d_6 \le d_5 \le (n-1), \\
& 3 \le d_4 \le  \cdots \leq
d_{1}\leq (n-1)\}
\end{align*}
\begin{align*}
\mathcal{S}_{5}^{3}=\{\mathbf{d} \in  \Sigma_{2(n+4)}: \,\,
&  1\leq d_{n}\leq \cdots \leq d_{6}\leq (n-1),\,\,\, 3\leq d_{5}\leq d_4 \le (n-1), \\
& 4 \le d_3 \le d_2 \le d_{1}\leq (n-1)\}
\end{align*}
while for $n=5$ only the last one and for $n=6$ the second and the third sets.

The maximal elements, for $n \ge 7$, are respectively,
\begin{equation*}
\mathbf{x^{\ast }}(\mathcal{S}_{5}^{1})=\left[ (n-1),6,2^5,1^{n-7}\right]
\end{equation*}%
\begin{equation*}
\mathbf{x^{\ast }}(\mathcal{S}_{5}^{2})=\left[ (n-1),5,3^2,2^2,1^{n-6}\right]
\end{equation*}%
\begin{equation*}
\mathbf{x^{\ast }}(\mathcal{S}_{5}^{3})=\left[ (n-1),4^2,3^2,1^{n-5}\right]
\end{equation*}%

Notice that the three degree sequences are not comparable in the majorization
order. Thus in case of pentacyclic graphs we have three degree sequences which correspond to the graphs $G^{*}(\mathcal{S}_{5}^{1})$, $G^{*}(\mathcal{S}_{5}^{2})$ and $G^{*}(\mathcal{S}_{5}^{3})$ in Figure \ref{fig:1} drawing for the sake of simplicity for $n=9$.

The minimal element of the set $\mathcal{S}_{5}^{1}$ is given by
\begin{equation}
\mathbf{x_{\ast }}(\mathcal{S}_{5}^{1})=\left[ \left (\frac{2(n+4)}{n} \right)^n \right] .
\end{equation}%
For  $n=8$ it  has integer components and thus $\mathbf{x_{\ast }}(\mathcal{S}_{5}^{1})=[3^8] $.

For $n=7$, by applying the procedure of Section 2, we find
as minimal element the vector $[4,3^6]$ while for $n>8$,  taking $k=2(n+4)-\Bigg\lfloor\dfrac{2(n+4)}{n}\Bigg\rfloor n=8$,
the minimal element with integer components is
\begin{equation*}
\mathbf{x_{\ast }^{\prime }}(\mathcal{S}_{5}^{1})=\left[ 3^8,2^{n-8}\right].
\end{equation*}

Simple computations show that the minimal element of the set $\mathcal{S}_{5}^{2}$, for $n\ge 7$
is the same as before.  Indeed, the minimal element of the set $\mathcal{S}_{5}^{1}$  belongs to  $\mathcal{S}_{5}^{2}$ and it is impossible to find an element of $\mathcal{S}_{5}^{2}$  minorizing  it.

For the set $\mathcal{S}_{5}^{3}$ and $n \ge 7$,  we cannot apply the same argument  as above since the minimal element of the set $\mathcal{S}_{5}^{1}$ does not belong to $\mathcal{S}_{5}^{3}$. By Theorem \ref{th:minimo} we get $k=5, d=0, \rho=2$. Thus the minimal element is
\begin{equation*}
\mathbf{x_{\ast }}(\mathcal{S}_{5}^{3})=\left[ 4^3,3^2, 2^{n-5}\right]
\end{equation*}
 Notice that $\left[ 3^8,2^{n-8}\right]   \trianglelefteq \left[ 4^3,3^2, 2^{n-5}\right]$ .

Thus in case of pentacyclic graphs we have one minimal degree sequence which corresponds to the graph $G_{*}(\mathcal{S}_{5}^{1})$ in Figure \ref{fig:2} drawing for the sake of simplicity for $n=9$.

Finally for $n=5$ there is only one pentacyclic graph associated to the degree sequence $[4^3,3^2]$ while for $n=6$ we get as maximal elements of  $\mathcal{S}_{5}^{2}$ and $\mathcal{S}_{5}^{3}$ the non comparable sequences $[5^2, 3^2,2^2]$ and $[5,4^2,3^2,1]$. The minimal elements are $[4^2,3^4]$ and   $[4^3,3^2,2]$ with $[4^2,3^4] \trianglelefteq [4^3,3^2,2]$.

\item {\bf Hexacyclic.} For $n \ge 8$ we should  consider five  different sets:
\begin{align*}
\mathcal{S}_{6}^{1}=\{\mathbf{d} \in  \Sigma_{2(n+5)}:\,\,
 1\leq d_{n}\leq \cdots \leq d_{9}\leq (n-1),\,\,\, 2\leq d_{8}\leq \cdots \leq
d_{1}\leq (n-1)\}
\end{align*}%
\begin{align*}
\mathcal{S}_{6}^{2}=\{\mathbf{d} \in  \Sigma_{2(n+5)}: \,\,
& 1\leq d_{n}\leq \cdots \leq d_{8}\leq (n-1),\,\,\,  2 \le d_7 \le d_6 \le d_5 \le (n-1), \\
 & 3 \le d_4 \le  \cdots \leq d_{1}\leq (n-1)\}
\end{align*}
\begin{align*}
\mathcal{S}_{6}^{3}=\{\mathbf{d} \in   \Sigma_{2(n+5)}:\,\,
&  1\leq d_{n}\leq \cdots \leq d_{7}\leq (n-1),\,\,\, 2 \leq d_6 \leq (n-1),  \\
 &3\leq d_{5} \leq d_4 \leq (n-1) ,\,\,\,  4 \le d_3 \le d_2 \le d_{1}\leq (n-1)\}
\end{align*}
\begin{align*}
\mathcal{S}_{6}^{4}=\{\mathbf{d} \in   \Sigma_{2(n+5)}:
 1\leq d_{n}\leq \cdots \leq d_{7}\leq (n-1),\,\,
3\leq d_{6}\leq \cdots \leq d_{1}\leq (n-1)\}
\end{align*}
\begin{align*}
\mathcal{S}_{6}^{5}=\{\mathbf{d} \in  \Sigma_{2(n+5)}:
1\leq d_{n}\leq \cdots \leq d_{6}\leq (n-1),\,\,\, 4\leq d_{5}\leq \cdots \le d_{1}\leq (n-1)\}
\end{align*}

while for $n=5$ only the last one, for $n=6$ the last three sets and for $n=7$ the last four sets.

The maximal elements, for $n \ge 8$, can be computed via Theorem \ref{th:massimo} and are respectively,
\begin{equation*}
\mathbf{x^{\ast }}(\mathcal{S}_{6}^{1})=\left[ (n-1),7,2^6,1^{n-8} \right]
\end{equation*}
\begin{equation*}
\mathbf{x^{\ast }}(\mathcal{S}_{6}^{2})=\left[ (n-1),6,3^2,2^3,1^{n-7}\right]
\end{equation*}
\begin{equation*}
\mathbf{x^{\ast }}(\mathcal{S}_{6}^{3})=\left[ (n-1),5, 4,3^2,2, 1^{n-6}\right]
\end{equation*}%
\begin{equation*}
\mathbf{x^{\ast }}(\mathcal{S}_{6}^{4})=\left[ (n-1),5,3^4,1^{n-6}\right]
\end{equation*}
\begin{equation*}
\mathbf{x^{\ast }}(\mathcal{S}_{6}^{5})=\left[ (n-1),4^4,1^{n-5}\right]
\end{equation*}%

Notice that the first three degree sequences and the last one are not comparable in the majorization
order, while $\mathbf{x^{\ast }}(\mathcal{S}_{6}^{4}) \trianglelefteq \mathbf{x^{\ast }}(\mathcal{S}_{6}^{3})$.  Thus in case of hexacyclic graphs we have four degree sequences which correspond to the graphs $G^{*}(\mathcal{S}_6^1)$, $G^{*}(\mathcal{S}_6^2)$, $G^{*}(\mathcal{S}_6^3)$, and $G^{*}(\mathcal{S}_6^5)$ in Figure \ref{fig:1} drawing for the sake of simplicity for $n=11$.

The minimal element of the set $\mathcal{S}_{6}^{1}$ is given by
\begin{equation}
\mathbf{x_{\ast }}(\mathcal{S}_{6}^{1})=\left[ \left (\frac{2(n+5)}{n} \right)^n \right] .
\end{equation}%
It has integer components for $n=10$ and the minimal element is $\mathbf{x_{\ast }}(\mathcal{S}_{6}^{1})=[3^{10}] $.

For $n=8$ and $n=9$, by applying the procedure of Section 2, we find
as minimal element the vectors $[4^2,3^6]$ and $[4,3^8]$, respectively.

Finally,  for $n>10$,  being  $k=2(n+5)-\Bigg\lfloor\dfrac{2(n+5)}{n}\Bigg\rfloor n=10$,
the minimal element with integer components is
\begin{equation*}
\mathbf{x_{\ast }^{\prime }}(\mathcal{S}_{6}^{1})=\left[ 3^{10},2^{n-10}\right].
\end{equation*}

Simple computations show that the minimal elements of the sets $\mathcal{S}_{6}^{2}$ and $\mathcal{S}_{6}^{4}$, for $n\ge 8$
 are the same as before.
 Indeed, the minimal element of the set $\mathcal{S}_{6}^{1}$  belongs to  both sets $\mathcal{S}_{6}^{2}$ and $\mathcal{S}_{6}^{4}$ and it is impossible to find an element of  the sets  minorizing  it.

For the sets $\mathcal{S}_{6}^{3}$ and $\mathcal{S}_{6}^{5}$ and $n \ge 8$,  we cannot apply the same argument  as above since the minimal element of the set $\mathcal{S}_{6}^{1}$ does not  belong to $ \mathcal{S}_{6}^{3}$ and $\mathcal{S}_{6}^{5}$.

For the first case, by Theorem \ref{th:minimo} we get $k=5, d=0, \rho=\left (\frac{2n-8}{n-5}\right )^{n-5}.$ After computing the minimal vector with  integer components we get
\begin{equation*}
\mathbf{x_{\ast }}(\mathcal{S}_{6}^{3})=\left[ 4^3,3^4, 2^{n-7}\right]
\end{equation*}
Notice that $\left[ 3^{10},2^{n-10}\right]   \trianglelefteq \left[ 4^3,3^4, 2^{n-7}\right]$ .

For the second case,  by Theorem \ref{th:minimo} we get $k=6, d=0, \rho=2.$ Thus the minimal element is
\begin{equation*}
\mathbf{x_{\ast }}(\mathcal{S}_{6}^{5})=\left[ 4^5, 2^{n-5}\right]
\end{equation*}
Notice that $\left[ 3^{10},2^{n-10}\right]   \trianglelefteq \left[ 4^5,  2^{n-5}\right]$.

Thus in case of hexacyclic graphs we have one minimal degree sequence which corresponds, for $n=10$,  to the graph $G_{*}(\mathcal{S}_6^1)$ in Figure \ref{fig:2} drawing for the sake of simplicity for $n=11$.

Finally,  for $n=5$ there is only one hexacyclic  graph associated to the degree sequence $[4^5]$.

For $n=6$ we get as maximal elements of  $\mathcal{S}_{6}^{3}$ and $\mathcal{S}_{6}^{5}$ the sequences $[5^2,4,  3^2,2]$ and $[5,4^4,1]$ which are non comparable with respect to the majorization order, while for $\mathcal{S}_{6}^{4}$  we get the sequence $[5^2,3^4]$ majorized by   $[5^2,4,  3^2,2]$.

The minimal element of  $\mathcal{S}_{6}^{3}$ and $\mathcal{S}_{6}^{4}$ is  $[4^4,3^2]$ while the minimal element of  $\mathcal{S}_{6}^{5}$  is $[4^5,2]$ which is minorized by $[4^4,3^2]$.

For $n=7$ we get as maximal elements of $\mathcal{S}_{6}^{2}$ , $\mathcal{S}_{6}^{3}$ and $\mathcal{S}_{6}^{5}$ the sequences $[6^2,3^2,2^3]$, $[6,5,4,  3^2,2,1]$ and $[6,4^4,1,1]$ which are non comparable with respect to the majorization order, while for $\mathcal{S}_{6}^{4}$  we get the sequence $[6,5,3^4,1]$ majorized by $[6,5,4,  3^2,2,1]$.

The minimal element of  $\mathcal{S}_{6}^{2}$, $\mathcal{S}_{6}^{3}$ and $\mathcal{S}_{6}^{4}$ is  $[4^3,3^4]$ while the minimal element of  $\mathcal{S}_{6}^{5}$  is $[4^5,2^2]$ which is minorized by $[4^3,3^4]$.

\end{itemize}
\newpage
In the following table we summarize  the maximal and minimal degree sequences of $c$-cyclic graphs for $0 \le c \le 6$ and $n \ge c+2$.

\begin{table}[!h]
\centering
\label{tab:Table}
\renewcommand\arraystretch{1.5}
\begin{tabular}{|r| l| l|}
%\hline
\hline \textbf{ c }&\textbf{Maximal degree sequences } &\textbf{Minimal degree sequences } \\ \hline
0 & $[(n-1),1^{n-1}]$ & $[2^{n-2},1^2]$ \\ \hline
1  & $[(n-1), 2^2, 1^{n-3}]$ & $[2^{n}]$  \\ \hline
2  & $[(n-1),3, 2^2, 1^{n-4}]$ & $[3^2, 2^{n-2}]$  \\ \hline
3  & $[(n-1),4, 2^3, 1^{n-5}], \,\, [(n-1),3^3,1^{n-4}]$ & $[3^4, 2^{n-4}]$  \\ \hline
4 & $[(n-1),5, 2^4, 1^{n-6}], \,\, [(n-1),4, 3^2,2, 1^{n-5}] $ & $[3^6, 2^{n-6}]$  \\ \hline
\multirow{2}{*}{5} & $[(n-1),6, 2^5, 1^{n-7}], \,\, [(n-1),5, 3^2,2^2, 1^{n-6}], \,\,$ & $[4,3^6]\,\,  (n=7), \,\,  [3^8, 2^{n-8}]\,\, (n \ge 8)$ \\
                 & $[(n-1), 4^2,3^2,1^{n-5}]$ &  \\ \hline
\multirow{2}{*}{6} & $[(n-1),7, 2^6,1^{n-8}], \,\, [(n-1),6, 3^2,2^3,1^{n-7}],\,\, $ & $[4^2,3^6] \,\, (n=8),\,\, [4,3^8] \,\, (n=9), \,\,$  \\
                  & $[(n-1),5,4,3^2,2,1^{n-6}],\,\, [(n-1),4^4,1^{n-5}] $ & $[3^{10}, 2^{n-10}]\,\, (n \ge 10)$ \\ \hline
%\hline
\end{tabular}
\caption{Extremal degree sequences}
\end{table}

\begin{remark}
We observe that the maximal and minimal degree sequences of c-cyclic graphs, for $0 \le c \le 6$ and $n \ge c+2$, can be expressed in terms  of the cyclomatic number c. In fact
\begin{itemize}
\item For every $c \ge 0$  there is the maximal element
$$
[(n-1), (c+1), 2^c, 1^{n-c-2}]
$$
\item For $c \ge 3$ there is a second  maximal element non comparable with the previous one:
$$
[(n-1), c, 3^2, 2^{c-3}, 1^{n-c-1}]
$$
\item For $c \ge 5$ there is a third  maximal element non comparable with the previous ones:
$$
[(n-1), c-1, 4, 3^2, 2^{c-5}, 1^{n-c}]
$$
\item For every c such that $2c-2 \le n$,  the minimal element is
$$
[3^{2c-2},2^{n-2c+2}]
$$
\end{itemize}
One may conjecture that these properties  hold also  for $c>6$.

\end{remark}

\section{Upper and lower bounds of degree-based topological indices.}

In this Section we  discuss  upper and lower bounds for degree-base topological indices  for $c$-cyclic graphs, $0 \le c \le 6$, via majorization techniques. To this aim, we will apply the characterizations of $c$-cyclic graphs provided in Theorem \ref{Schocker1} and the extremal elements computed in Section 4.

We recall that a symmetric
function $\phi: A\rightarrow
\mathbb{R}$, $A\subseteq
\mathbb{R}^{n}$, is said to be Schur-convex on $A$ if $\mathbf{x}\trianglelefteq
\mathbf{y}$ implies $\phi (\mathbf{x})\leq $ $\phi (\mathbf{y})$.
\noindent Given an interval $I\subset
%TCIMACRO{\U{211d} }%
%BeginExpansion
\mathbb{R}
%EndExpansion
$, and a  convex function $g:I\rightarrow
%TCIMACRO{\U{211d} }%
%BeginExpansion
\mathbb{R}
%EndExpansion
$, the function $\phi (\mathbf{x})=\sum_{i=1}^{n}g(x_{i})$ is
Schur-convex on $I^{n}=\underbrace{I\times I \times \cdots \times I}_{n-
times}$.

\noindent If $\phi$ is a Schur-convex function and $S$ is a subset of $\Sigma_a$ which admits maximal and minimal elements with respect to the majorization order,  the solutions of the  constrained optimization problems
\begin{equation}  \label{problema P}
\begin{cases}
\max \,\,\,(\min )\text{ }\phi (\mathbf{x})\tag{P} \\
\text{subject to }\mathbf{x}\in S
\end{cases}
\end{equation}
are the maximal element $x^*(S)$ and the minimal element $x_*(S)$, respectively. On the other hand, when $\phi$ is a Schur-concave function, the solutions to problem \eqref{problema P} are the minimal element $x_*(S)$ and the maximal element $x^*(S)$, respectively.

Let now $F(d_1,d_2, \cdots, d_n)$ be any topological index which is a Schur-convex  function of its arguments,  defined on a subset $S \subseteq \Sigma_a$. Then by the order preserving property of the Schur-convex functions, we get
$$
F(x_*(S)) \le F(d_1,d_2, \cdots, d_n) \le F(x^*(S))
$$
Analogously, if $F$ is a Schur-concave function, then
$$
F(x^*(S)) \le F(d_1,d_2, \cdots, d_n) \le F(x_*(S)).
$$
A class of topological indices,  of particular interest found in the literature and depending on the degree sequence of a graph over all vertices,   are the {\it first general Zagreb indices} defined as
\begin{equation} \label{eq:indice1}
M^{\alpha}_1 = \sum_{i=1}^n d_i^{\alpha}
\end{equation}
where $\alpha$ is an arbitrary real number except $0$ and $1$ (see \cite{LiZheng}). For $\alpha=2$ we get the {\it first Zagreb index} while for $\alpha=-1$ the {\it inverse degree}. Notice that $M^{\alpha}_1$ is a Schur convex function of the degree sequence either for $\alpha <0$ or $\alpha >1$ while it is a Schur concave function of the degree sequence for $0<\alpha <1$. Also the {\it first multiplicative Zagreb index}
\begin{equation} \label{eq:indice2}
\ln M_1= 2\sum_{i=1}^n \ln (d_i)
\end{equation}
introduced by Gutman in \cite{Gutman2011}, is a Schur concave function of the degree sequence.

For  $c$-cyclic graphs, $0\le c \le 6$,  the upper and lower bounds of   indices \eqref{eq:indice1}, \eqref{eq:indice2}, and in general of any index which is a Schur convex/concave function of the degree sequence, can be found throughout the solution of (P) where $S$ in any of the subsets $\mathcal{S}_i^j$ described in Section 4. From these arguments  the upper and lower bounds in Table 2 follow.
For convenience we restrict our analysis to the first general Zagreb index with either $\alpha <0$ or $\alpha >1$ and $n \ge c+2$. In the case $0<\alpha <1$ the upper and lower bounds are turned over. Bounds for the first multiplicative Zagreb index can be obtained analogously.

\begin{table}[!h]
\centering
\label{tab:Table2}
\renewcommand\arraystretch{1.5}
\begin{tabular}{|l|l|l|}
\hline
$\mathbf{c}$ & \textbf{Lower bounds} & \textbf{Upper bounds} \\ \hline
$1$ & $(n-1)^{\alpha }+2^{\alpha +1}+(n-3)$ & $ n\left(
2^{\alpha }\right) $ \\ \hline
$2$ & $(n-1)^{\alpha }+3^{\alpha }+2^{\alpha +1}+(n-4)$ &
$2\left( 3^{\alpha }\right) +\left( 2^{\alpha }\right) \left( n-2\right) $
\\ \hline
\multirow{2}{*}{$3$} & \multirow{2}{*}{$4\left( 3^{\alpha }\right) +\left( 2^{\alpha }\right) \left( n-4\right) $}
& $(n-1)^{\alpha }+4^{\alpha }+3\left( 2^{\alpha }\right) +(n-5)$ \\

 & & $(n-1)^{\alpha }+3^{\alpha +1}+\left( n-4\right)$

\\ \hline
\multirow{2}{*}{$4$} & \multirow{2}{*}{$2\left( 3^{\alpha +1}\right) +\left(
2^{\alpha }\right) \left( n-6\right) $}

& $(n-1)^{\alpha }+5^{\alpha }+2^{\alpha +2}+\left( n-6\right)$ \\

 & & $(n-1)^{\alpha }+2^{\alpha }({2^\alpha +1 })+2\left( 3^{\alpha }\right) +\left(
n-5\right)$
\\ \hline
\multirow{3}{*}{$5$} & \multirow{3}{*}{$ 8\left( 3^{\alpha }\right) +\left( n-8\right) 2^{\alpha }$}
&
$(n-1)^{\alpha }+6^{\alpha }+5\left( 2^{\alpha }\right) +(n-7) $ \\
& &  $(n-1)^{\alpha }+5^{\alpha }+2\left( 3^{\alpha }\right) +2\left( 2^{\alpha
}\right) +(n-6)$\\
& & $ (n-1)^{\alpha }+2^{2\alpha +1}+2\left( 3^{\alpha }\right) +(n-8) $

\\ \hline
\multirow{4}{*}{$6$} & \multirow{4}{*} {$ 10\left(3^{\alpha }\right) +\left( n-10\right) 2^{\alpha }$}
&
$(n-1)^{\alpha }+7+6\left( 2^{\alpha }\right) (n-8)$  \\

& & $ (n-1)^{\alpha }+6+2\left( 3^{\alpha }\right) +3\left( 2^{\alpha }\right)
+(n-7)$  \\
& & $ (n-1)^{\alpha }+5+4+2\left( 3^{\alpha }\right) +2+(n-6)$ \\
& & $ (n-1)^{\alpha }+2^{2\alpha+2}+(n-5) $
\\ \hline
\end{tabular}
\caption{Bounds for $M^{\alpha}_1$ ($\alpha <0 \vee \alpha >1$)}
\end{table}

Note that in the case of bicyclic graphs, for the  bounds of $M_1^{\alpha}$ we recover the same results as in \cite{Zhang2006}, Theorems 1, 5, 7 and 8.

Notice that when more maximal elements are identified, the best choice depends on $\alpha.$ We discuss in detail the case $\alpha=-1$, which gives  the \emph{inverse degree} $\rho
=M_{1}^{-1}=\sum {\frac{1}{{d_{j}}}}$ (see \cite{LiShi}, \cite{Danke1} and \cite{Danke2}).

Let $G$ be a $c-$cyclic graph $(0\leq
c\leq 6)$.

\begin{itemize}
\item {Tree }$\ {(}c=0)${:}
\[
\dfrac{n+2}{2}\leq \rho \leq (n-1)+\dfrac{1}{n-1},
\]

where the lower and the upper bounds are attained if and only if $G=P_{n}$
and $G=S_{n}$ respectively, being $P_{n}$ and $S_{n}$ the path and the star
with $n$ vertices.

\item {Unicyclic } $(c=1)${:}

\[
\dfrac{n}{2}\leq \rho \leq (n-2)+\dfrac{1}{n-1},
\]
where the lower and the upper bounds are attained if and only if $G=C_{n}$
and $G=S_{n}^{+}$ respectively, being $C_{n}$ the cycle and $S_{n}^{+}$ the
graph obtained from $S_{n}$ joining two pendant vertices with an edge.

It is worth to point out that for trees and unicyclic graphs we recover the
same results as in \cite{LiShi}, Theorems 3 and 4, respectively.

\item {Bicyclic }$(c=2)${: }
\[
\dfrac{n-2}{2}+\dfrac{2}{3}\leq \rho \leq (n-3)+\dfrac{1}{n-1}+\dfrac{1}{3},
\]%
where the upper bound is attained by the graphs $H_{n}^{1}$ in \cite%
{Zhang2006}, the lower by the graph in the class $G_{n}^{1}$ in \cite%
{Zhang2006}.

For $c\ge 3$ we take into account the upper bounds corresponding to the
largest values that can be attained. We consider the case $n \ge c+2$.

\item {Tricyclic }$(c=3)${: }

\[
\dfrac{n-4}{2}+\dfrac{4}{3}\leq \rho \leq (n-3)+\dfrac{1}{n-1}.
\]

The graphs $G_{*}(\mathcal{S}_3^1)$ in Figure 2 and $G^{*}(\mathcal{S}_3^2)$ in Figure 1 achieve, for $n=8,$  the
lower and the upper bounds, respectively.

\item {Tetracyclic }$(c=4)${: }

\[
\dfrac{n-2}{2}\leq \rho \leq (n-5)+\dfrac{1}{n-1}+\dfrac{17}{12}.
\]

The graphs $G_{*}(\mathcal{S}_4^1)$ in Figure 2 and $G^{*}(\mathcal{S}_4^2)$ in Figure 1 achieve, for $n=8,$ the
lower and the upper bounds, respectively.

\item {Pentacyclic }$(c=5)${: }
\[
\dfrac{n-8}{2}+\dfrac{8}{3}\leq \rho \leq (n-5)+\dfrac{1}{n-1}+\dfrac{7}{6}.
\]

The graphs $G_{*}(\mathcal{S}_5^1)$ in Figure 2 and $G^{*}(\mathcal{S}_5^2)$ in Figure 1 achieve, for $n=9,$  the
lower and the upper bounds, respectively.

\item Hexacyclic $(c=6)$:
\[
\dfrac{n-10}{2}+\dfrac{10}{3}
\leq \rho \leq (n-4)+\dfrac{1}{n-1}.
\]

The graphs $G_{*}(\mathcal{S}_6^1)$ in Figure 2 and $G^{*}(\mathcal{S}_6^5)$ in Figure 1 achieve, for $n=11,$  the
lower and the upper bounds, respectively.

We observe that for $c \ge 3 $ if the additional information $d_{c+2}\geq 2$ is known, the upper bound can be improved  yielding:
\[
\rho \leq (n-c)+\dfrac{1}{n-1}+\dfrac{c^{2}-3c-2}{2(c+1)}
\]
which is  attained by the graphs $G^{*}(\mathcal{S}_c^1)$.

\end{itemize}

\section{Conclusion}
In this paper we focus on lower and upper bounds \ of some relevant graph
topological indices, based on the degree sequence over all vertices of the
graph.  We get our results  through new characterizations of $c$-cyclic
graphs aimed to identify  extremal vectors with respect to the majorization
order of particular subset of $%
%TCIMACRO{\U{211d} }%
%BeginExpansion
\mathbb{R}
%EndExpansion
^{n}.$ We have shown that classical bounds  can be recovered and new ones
can be obtained.  Our results suggest that as well as the generalization of Theorem 5 to $c >6$ can be found,
 our theoretical approach can be extended also to a
larger class of $c$-cyclic graphs.  Finally, other topological indices of $c$%
-cyclic graphs can be provided  whenever they can be expressed as
Schur-convex or Schur-concave functions of the degree sequence of the graph.

\newpage

\begin{figure}[!h]
\centering
\fboxrule=2pt
\subfigure[\scriptsize{$G^{*}(\mathcal{S}_3^1)=${$[7,4,2^3,1^3]$}}]{\includegraphics[height=4.5cm]{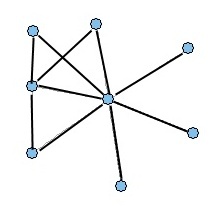}}
\subfigure[\scriptsize{$G^{*}(\mathcal{S}_3^2)=${$[7,3^3,1^4]$}}]{\includegraphics[height=4.5cm]{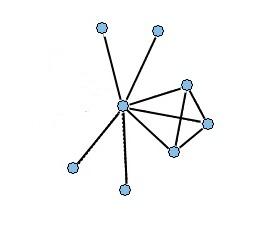}}
\subfigure[\scriptsize{$G^{*}(\mathcal{S}_4^1)=${$[7,5,2^4,1^2]$}}]{\includegraphics[height=4.5cm]{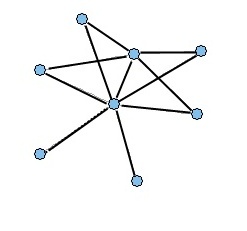}}
\subfigure[\scriptsize{$G^{*}(\mathcal{S}_4^2)=${$[7,4,3^2,2,1^3]$}}]{\includegraphics[height=4.5cm]{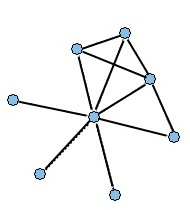}}
\subfigure[\scriptsize{$G^{*}(\mathcal{S}_5^1)=${$[8,6,2^5,1^2,]$}}]{\includegraphics[height=4.5cm]{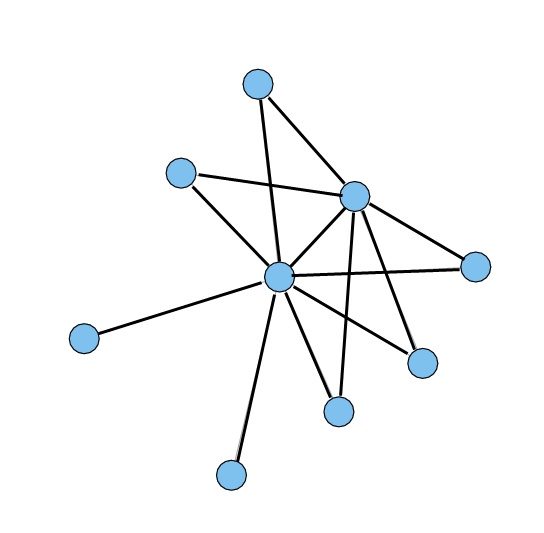}}
\subfigure[\scriptsize{$G^{*}(\mathcal{S}_5^2)=${$[8,5,3^2,2^2,1^3]$}}]{\includegraphics[height=4.5cm]{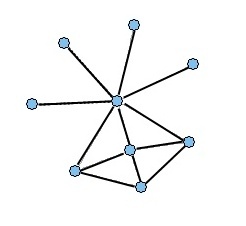}}
\subfigure[\scriptsize{$G^{*}(\mathcal{S}_5^3)=${$[8,4^2,3^2,1^4]$}}]{\includegraphics[height=4.5cm]{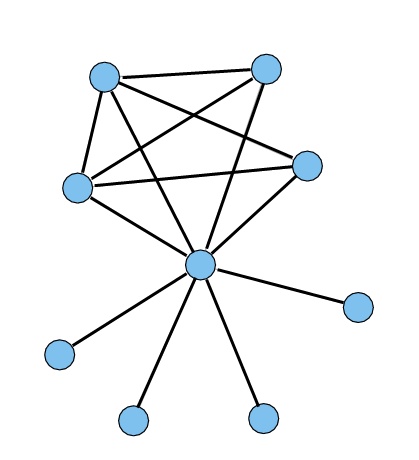}}
\subfigure[\scriptsize{$G^{*}(\mathcal{S}_6^1)=${$[10,7,2^6,1^3]$}}]{\includegraphics[height=4.5cm]{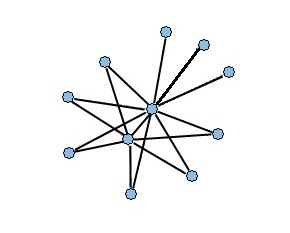}}
\subfigure[\scriptsize{$G^{*}(\mathcal{S}_6^2)=${$[10,6,3^2,2^3,1^4]$}}]{\includegraphics[height=4.5cm]{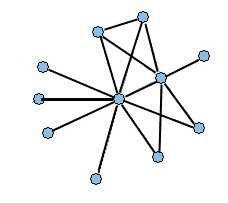}}
\subfigure[\scriptsize{$G^{*}(\mathcal{S}_6^3)=${$[10,5,4,3^2,2,1^5]$}}]{\includegraphics[height=4.5cm]{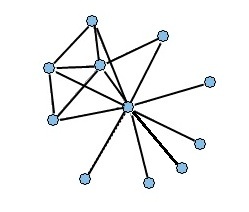}}
%\subfigure[\scriptsize{$G(\mathcal{S}_6^4)=${$[9,5,3,3,3,3,1,1,1,1]$}}]{\includegraphics[height=4.5cm]{S6_4.jpeg}}
\subfigure[\scriptsize{$G^{*}(\mathcal{S}_6^5)=${$[10,4^4,1^6]$}}]{\includegraphics[height=4.5cm]{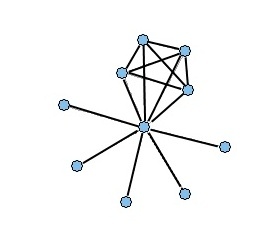}}
\caption{Maximal degree sequence graphs}
\label{fig:1}
\end{figure}

\newpage

\begin{figure}[!h]
\centering
\subfigure[\scriptsize{$G_{*}(\mathcal{S}_3^1)=${$[3^4,2^4]$}}]{\includegraphics[height=4.5cm]{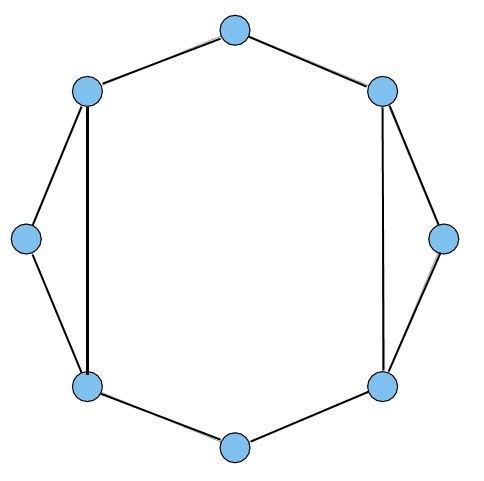}}
\subfigure[\scriptsize{$G_{*}(\mathcal{S}_4^1)=${$[3^6,2^2]$}}]{\includegraphics[height=4.5cm]{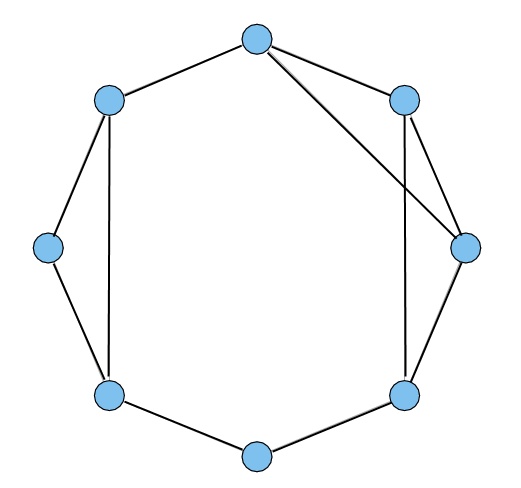}}
%\subfigure[\scriptsize{$G(\mathcal{S})=${$[4,3,3,3,3,3,3]$}}]{\includegraphics[height=4.5cm]{S_5n=7.jpeg}}
\subfigure[\scriptsize{$G_{*}(\mathcal{S}_5^1)=${$[3^8,2]$}}]{\includegraphics[height=4.5cm]{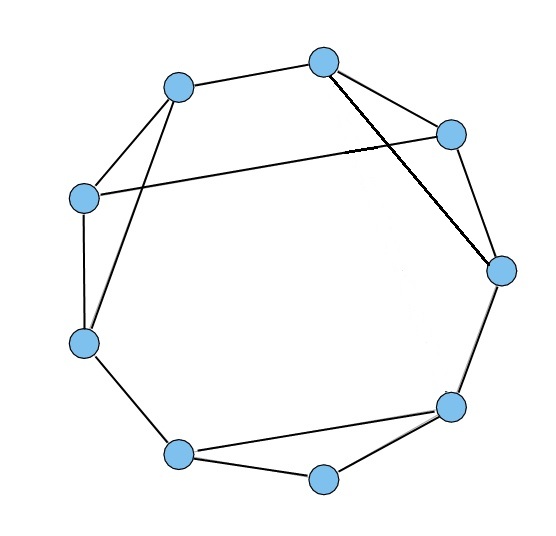}}
%\subfigure[\scriptsize{$G(\mathcal{S})=${$[4,4,3,3,3,3,3,3]$}}]{\includegraphics[height=4.5cm]{S_6(n=8).jpeg}}
%\subfigure[\scriptsize{$G(\mathcal{S})=${$[4,3,3,3,3,3,3,3,3]$}}]{\includegraphics[height=4.5cm]{S_6(n=9).jpeg}}
\subfigure[\scriptsize{$G_{*}(\mathcal{S}_6^1)=${$[3^{10},2]$}}]{\includegraphics[height=4cm]{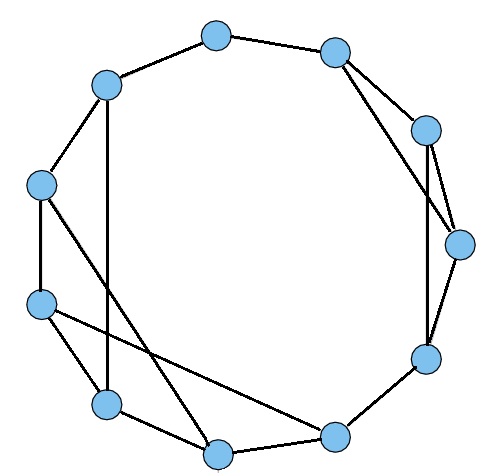}}
\caption{Minimal degree sequence graphs}
\label{fig:2}
\end{figure}

\end{document}